\documentclass{article}
\usepackage{graphicx} 
\usepackage{amsmath, amssymb, amsthm}
\usepackage{enumitem}
\usepackage{booktabs}
\usepackage{longtable}
\usepackage{hyperref}
\usepackage{tikz}
\usepackage{tikz-cd}
\usepackage{float}
\usepackage{textcomp}
\newcommand{\cent}{\textcentoldstyle}

\hypersetup{colorlinks=true, linkcolor=blue!60!black,
            citecolor=blue!60!black}

\newtheorem{theorem}{Theorem}[section]
\newtheorem{proposition}[theorem]{Proposition}

\newtheorem{lemma}[theorem]{Lemma}
\theoremstyle{definition}
\newtheorem{definition}[theorem]{Definition}
\newtheorem{remark}[theorem]{Remark}

\newcommand{\ic}{\mathrm{ic}}

\title{A Formal Graph-Theoretic Framework for Pitch Class Set Analysis}
\author{Aleksa Joksimović}
\date{May 2026}

\begin{document}

\maketitle

\begin{abstract}
    We present a graph-theoretic reformulation of pitch-class set theory in which each set in $\mathbb{Z}_n$ is represented as a complete weighted graph whose edge weights are interval classes. We show that this construction is invariant under the dihedral group $D_n$, and that the full interval structure is encoded by a cyclic step composition, from which all interval data are recovered via an additivity principle. This framework yields a direct correspondence between T/I equivalence and graph isomorphism, and reinterprets Z-relation as non-isomorphic graphs with identical edge-weight multisets. We extend the model to weighted clique complexes, linking higher-order homometry to simplex-weight structure, and introduce a cent-weighted formulation enabling comparisons across different equal temperaments. Finally, we define a polynomial invariant derived from antipodal step pairings for algebraic analysis of pitch class space.

\end{abstract}
\section{Introduction}
This paper proposes a combinatorial, graph-theoretic reformulation of pitch class set theory introduced by Allen Forte \cite{forte1973}. Although methods of graph and network theory have been applied to musical analysis before \cite{alcala2018}, we specifically use them to model the combinatorial properties of pitch class sets under the dihedral group $D_n$ across arbitrary equal temperaments. The paper assumes a certain level of familiarity with the fundamental concepts of pitch class set theory and is structured as a catalogue of results, rather than an in-depth extensive narrative.

\vspace{1em}
\begin{definition}[\textit{Pitch class graph}]
\label{def:musical-graph}
Let $P \subseteq \mathbb{Z}_n$ be a pitch class set. The
\textbf{pitch class graph} of $P$ is the complete weighted graph
$G(P) = (V, E, w)$ where $V = P$, $E$ is the set of all unordered
pairs of distinct elements of $P$, and $w(\{p, q\}) = \ic(p, q)$,
where
\[
    \ic(p, q) = \min(|p - q|,\ n - |p - q|)
    \;\in\; \{1, \ldots, \lfloor n/2 \rfloor\}.
\]
\end{definition}

The dihedral group $D_n$ acts on $\mathbb{Z}_n$ by transpositions
$T_s(p) = p + s$ and inversions $I_s(p) = s - p$ (mod $n$). Since
$\ic(g(p), g(q)) = \ic(p, q)$ for all $g \in D_n$, T/I-equivalent
sets produce isomorphic pitch class graphs.

\begin{lemma}
\label{lem:dihedral}
$\ic(g(p), g(q)) = \ic(p, q)$ for all $g \in D_n$.
\end{lemma}

\begin{proof}
For transposition: $|T_s(p) - T_s(q)| = |(p+s)-(q+s)| = |p-q|$.
For inversion: $|I_s(p) - I_s(q)| = |(s-p)-(s-q)| = |p-q|$.
In both cases $\ic$ is preserved.
\end{proof}

Valid pitch class graphs must satisfy the following two properties:

\begin{definition}[Closure]
\label{def:closure}
Let $P \subseteq \mathbb{Z}_n$ have cardinality $k$, written in
ascending order with $p_0 = 0$. The \textbf{steps} of $P$ are
\[
    s_i = p_i - p_{i-1} \quad (i = 1, \ldots, k-1),
    \qquad s_k = n - p_{k-1}.
\]
The steps form a \textbf{composition} of $n$ --- an ordered
partition of $n$ into $k$ positive parts --- satisfying the
\textbf{closure property}:
\[
    s_1 + s_2 + \cdots + s_k = n.
\]
\end{definition}
This is a direct concequence of $\mathbb{Z}_n$ arithmetic.
\begin{definition}[Additivity]
\label{def:additivity}
Once the steps of $P$ are fixed, every interval class is determined.
The interval class between adjacent pitch classes is
$\ic(p_{i-1}, p_i) = \min(s_i, n - s_i)$.
The interval class between any two pitch classes $p_i$ and $p_j$
with $j > i$ is:
\[
    \ic(p_i, p_j) = \min\!\left(\sum_{\ell=i+1}^{j} s_\ell,\;\;
    n - \sum_{\ell=i+1}^{j} s_\ell\right).
\]
Non-adjacent interval classes are called \textbf{diagonals} and are
each determined by a single addition and a single $\min$. This is
the \textbf{additivity rule}: the interval class between any two
pitch classes equals the interval class of the sum of the steps
between them.
\end{definition}

The collection of all $\binom{k}{2}$ interval classes of $P$ is the
\textbf{interval multiset} $\mu(P)$, equivalently the interval
vector written as a multiset. Since every interval class is
determined by the composition $(s_1, \ldots, s_k)$, the entire
interval structure of $P$ is encoded in its composition of $n$.

Figure~\ref{fig:trichord-realizations} shows all six pitch class graphs
realizing the interval multiset $\{3,4,5\}$ in $\mathbb{Z}_{12}$.
These form the complete T/I orbit of $\{0,3,7\}$: transposition and
inversion act as vertex relabelings preserving every edge weight, so
all six graphs are isomorphic and the T/I class is represented by
the single unlabeled triangle in Figure~\ref{fig:canonical-triangle}.

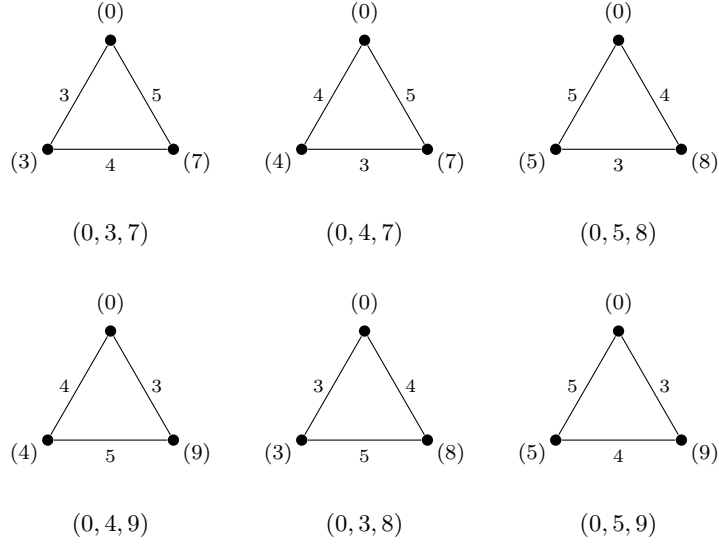
\begin{figure}[H]
\centering
\begin{tikzpicture}[scale=0.8]
\begin{scope}[xshift=0cm]
  \node (a0) at (90:1.2)  [circle,fill=black,inner sep=1.5pt]{};
  \node at (90:1.65)  {\footnotesize $(0)$};
  \node (a1) at (210:1.2) [circle,fill=black,inner sep=1.5pt]{};
  \node at (210:1.65) {\footnotesize $(3)$};
  \node (a2) at (330:1.2) [circle,fill=black,inner sep=1.5pt]{};
  \node at (330:1.65) {\footnotesize $(7)$};
  \draw (a0)--(a1) node[midway,left]  {\scriptsize $3$};
  \draw (a1)--(a2) node[midway,below] {\scriptsize $4$};
  \draw (a2)--(a0) node[midway,right] {\scriptsize $5$};
  \node at (0,-2.0) {\small $(0,3,7)$};
\end{scope}
\begin{scope}[xshift=4.2cm]
  \node (b0) at (90:1.2)  [circle,fill=black,inner sep=1.5pt]{};
  \node at (90:1.65)  {\footnotesize $(0)$};
  \node (b1) at (210:1.2) [circle,fill=black,inner sep=1.5pt]{};
  \node at (210:1.65) {\footnotesize $(4)$};
  \node (b2) at (330:1.2) [circle,fill=black,inner sep=1.5pt]{};
  \node at (330:1.65) {\footnotesize $(7)$};
  \draw (b0)--(b1) node[midway,left]  {\scriptsize $4$};
  \draw (b1)--(b2) node[midway,below] {\scriptsize $3$};
  \draw (b2)--(b0) node[midway,right] {\scriptsize $5$};
  \node at (0,-2.0) {\small $(0,4,7)$};
\end{scope}
\begin{scope}[xshift=8.4cm]
  \node (c0) at (90:1.2)  [circle,fill=black,inner sep=1.5pt]{};
  \node at (90:1.65)  {\footnotesize $(0)$};
  \node (c1) at (210:1.2) [circle,fill=black,inner sep=1.5pt]{};
  \node at (210:1.65) {\footnotesize $(5)$};
  \node (c2) at (330:1.2) [circle,fill=black,inner sep=1.5pt]{};
  \node at (330:1.65) {\footnotesize $(8)$};
  \draw (c0)--(c1) node[midway,left]  {\scriptsize $5$};
  \draw (c1)--(c2) node[midway,below] {\scriptsize $3$};
  \draw (c2)--(c0) node[midway,right] {\scriptsize $4$};
  \node at (0,-2.0) {\small $(0,5,8)$};
\end{scope}
\begin{scope}[xshift=0cm,yshift=-4.8cm]
  \node (d0) at (90:1.2)  [circle,fill=black,inner sep=1.5pt]{};
  \node at (90:1.65)  {\footnotesize $(0)$};
  \node (d1) at (210:1.2) [circle,fill=black,inner sep=1.5pt]{};
  \node at (210:1.65) {\footnotesize $(4)$};
  \node (d2) at (330:1.2) [circle,fill=black,inner sep=1.5pt]{};
  \node at (330:1.65) {\footnotesize $(9)$};
  \draw (d0)--(d1) node[midway,left]  {\scriptsize $4$};
  \draw (d1)--(d2) node[midway,below] {\scriptsize $5$};
  \draw (d2)--(d0) node[midway,right] {\scriptsize $3$};
  \node at (0,-2.0) {\small $(0,4,9)$};
\end{scope}
\begin{scope}[xshift=4.2cm,yshift=-4.8cm]
  \node (e0) at (90:1.2)  [circle,fill=black,inner sep=1.5pt]{};
  \node at (90:1.65)  {\footnotesize $(0)$};
  \node (e1) at (210:1.2) [circle,fill=black,inner sep=1.5pt]{};
  \node at (210:1.65) {\footnotesize $(3)$};
  \node (e2) at (330:1.2) [circle,fill=black,inner sep=1.5pt]{};
  \node at (330:1.65) {\footnotesize $(8)$};
  \draw (e0)--(e1) node[midway,left]  {\scriptsize $3$};
  \draw (e1)--(e2) node[midway,below] {\scriptsize $5$};
  \draw (e2)--(e0) node[midway,right] {\scriptsize $4$};
  \node at (0,-2.0) {\small $(0,3,8)$};
\end{scope}
\begin{scope}[xshift=8.4cm,yshift=-4.8cm]
  \node (f0) at (90:1.2)  [circle,fill=black,inner sep=1.5pt]{};
  \node at (90:1.65)  {\footnotesize $(0)$};
  \node (f1) at (210:1.2) [circle,fill=black,inner sep=1.5pt]{};
  \node at (210:1.65) {\footnotesize $(5)$};
  \node (f2) at (330:1.2) [circle,fill=black,inner sep=1.5pt]{};
  \node at (330:1.65) {\footnotesize $(9)$};
  \draw (f0)--(f1) node[midway,left]  {\scriptsize $5$};
  \draw (f1)--(f2) node[midway,below] {\scriptsize $4$};
  \draw (f2)--(f0) node[midway,right] {\scriptsize $3$};
  \node at (0,-2.0) {\small $(0,5,9)$};
\end{scope}
\end{tikzpicture}
\caption{All six realizations of the interval multiset $\{3,4,5\}$
in $\mathbb{Z}_{12}$, forming the complete T/I orbit of $\{0,3,7\}$.}
\label{fig:trichord-realizations}
\end{figure}

\begin{figure}[H]
\centering
\begin{tikzpicture}[scale=0.9]
  \node (v0) at (90:1.6)  [circle,fill=black,inner sep=1.8pt]{};
  \node (v1) at (210:1.6) [circle,fill=black,inner sep=1.8pt]{};
  \node (v2) at (330:1.6) [circle,fill=black,inner sep=1.8pt]{};
  \draw (v0)--(v1) node[midway,left]  {\scriptsize $3$};
  \draw (v1)--(v2) node[midway,below] {\scriptsize $4$};
  \draw (v2)--(v0) node[midway,right] {\scriptsize $5$};
\end{tikzpicture}
\caption{Canonical weighted triangle representing the T/I class of
$\{0,3,7\}$ with interval multiset $\{3,4,5\}$.}
\label{fig:canonical-triangle}
\end{figure}
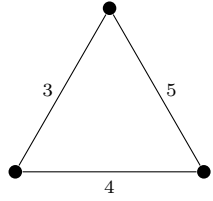

\begin{theorem}
\label{thm:cycle-determines-graph}
The pitch class graph $G(P)$ is uniquely and completely determined
by its cycle subgraph $C(P)$.
\end{theorem}

\begin{proof}
The cycle $C(P)$ encodes the step sequence $(w_0, w_1, \ldots,
w_{k-1})$ of $P$. By the additivity rule
(Definition~\ref{def:additivity}), every edge weight of $G(P)$ is
\[
    w(\{p_i, p_j\}) = \ic(p_i, p_j) =
    \min\!\left(\sum_{\ell=i}^{j-1} w_\ell,\;\;
    n - \sum_{\ell=i}^{j-1} w_\ell\right),
\]
a function of consecutive cycle weights alone. Since $G(P)$ is
complete, its edge set is exactly the set of all such pairs, and
every edge weight is determined by the above formula. Therefore
$C(P)$ determines $G(P)$ uniquely.
\end{proof}

\begin{definition}[Cent-weighting]
The
\textbf{cent-weighted pitch class graph} of $P$ is the complete weighted graph
$G(P) = (V, E, w')$ where $V = P \subseteq \mathbb{Z}_n$, $E$ is the set of all unordered
pairs of distinct elements of $P$ and
\[
    w'\{ p,q\} = \delta(n) \cdot\ic(p, q) = \frac{1200}{n}\min(|p - q|,\ n - |p - q|) \in (0,600]
    .
\]
\end{definition}

In $n$-TET, each step has size $\frac{1200}{n}$ cents - the \textbf{diesis} $\delta(n)$ of the temperament (the term diesis being loosely borrowed from music theory). This real interval is the logarithmic pitch class space that represents a common ambient space across all equal temperaments.

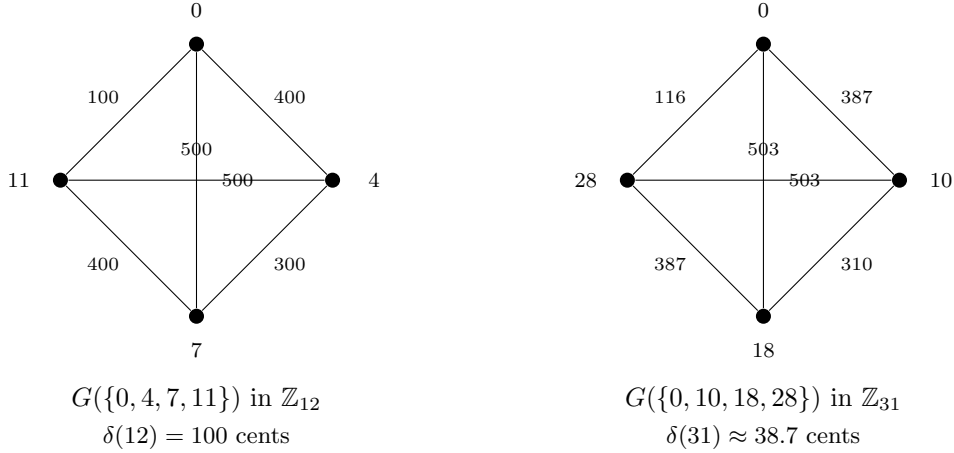
\begin{figure}[H]
\centering
\begin{tikzpicture}[scale=1.0]

\begin{scope}[xshift=0cm]
  \node (a0) at (90:1.8)   [circle,fill=black,inner sep=2pt]{};
  \node at (90:2.25)        {\footnotesize $0$};
  \node (a1) at (0:1.8)    [circle,fill=black,inner sep=2pt]{};
  \node at (0:2.35)         {\footnotesize $4$};
  \node (a2) at (270:1.8)  [circle,fill=black,inner sep=2pt]{};
  \node at (270:2.25)       {\footnotesize $7$};
  \node (a3) at (180:1.8)  [circle,fill=black,inner sep=2pt]{};
  \node at (180:2.35)       {\footnotesize $11$};

  \draw (a0)--(a1) node[midway,above right] {\scriptsize $400$};
  \draw (a0)--(a2) node[midway,right=6pt]   {\scriptsize $500$};
  \draw (a0)--(a3) node[midway,above left]  {\scriptsize $100$};
  \draw (a1)--(a2) node[midway,below right] {\scriptsize $300$};
  \draw (a1)--(a3) node[midway,above=6pt]   {\scriptsize $500$};
  \draw (a2)--(a3) node[midway,below left]  {\scriptsize $400$};

  \node at (0,-2.9) {$G(\{0,4,7,11\})$ in $\mathbb{Z}_{12}$};
  \node at (0,-3.4) {\small $\delta(12) = 100$ cents};
\end{scope}

\begin{scope}[xshift=7.5cm]
  \node (b0) at (90:1.8)   [circle,fill=black,inner sep=2pt]{};
  \node at (90:2.25)        {\footnotesize $0$};
  \node (b1) at (0:1.8)    [circle,fill=black,inner sep=2pt]{};
  \node at (0:2.35)         {\footnotesize $10$};
  \node (b2) at (270:1.8)  [circle,fill=black,inner sep=2pt]{};
  \node at (270:2.25)       {\footnotesize $18$};
  \node (b3) at (180:1.8)  [circle,fill=black,inner sep=2pt]{};
  \node at (180:2.35)       {\footnotesize $28$};

  \draw (b0)--(b1) node[midway,above right] {\scriptsize $387$};
  \draw (b0)--(b2) node[midway,right=6pt]   {\scriptsize $503$};
  \draw (b0)--(b3) node[midway,above left]  {\scriptsize $116$};
  \draw (b1)--(b2) node[midway,below right] {\scriptsize $310$};
  \draw (b1)--(b3) node[midway,above=6pt]   {\scriptsize $503$};
  \draw (b2)--(b3) node[midway,below left]  {\scriptsize $387$};

  \node at (0,-2.9) {$G(\{0,10,18,28\})$ in $\mathbb{Z}_{31}$};
  \node at (0,-3.4) {\small $\delta(31) \approx 38.7$ cents};
\end{scope}

\end{tikzpicture}
\caption{Cent-weighted pitch class graphs of the dominant seventh chord
$\{0,4,7,11\}$ in $\mathbb{Z}_{12}$ (left) and its closest
approximation $\{0,10,18,28\}$ in $\mathbb{Z}_{31}$ (right).
Edge weights are in cents, rounded to the nearest integer.
The two graphs are non-isomorphic as weighted graphs: the minor
second $100\,\cent$ in $\mathbb{Z}_{12}$ corresponds to
$116\,\cent$ in $\mathbb{Z}_{31}$, reflecting the distinct
dieses $\delta(12) = 100$ and $\delta(31) \approx 38.7$.}

\label{fig:dominant-seventh-cent}
\end{figure}

\section{Homometry and non-isomorphic graphs}

Pitch class sets that have the same interval class multiset (equivalently the same interval vector) but cannot be inverted or transposed into one another (meaning they form distinct T/I orbits) are called \textbf{homometric} or \textit{Z-related} sets. In our graph theoretic language, this translates to two non-isomorphic graphs with the same interval class multiset.

\begin{lemma}
\label{lem:iso-iff-ti}
$G(P_1) \cong G(P_2) \Leftrightarrow P_1 \sim_{T/I} P_2$.
\end{lemma}

\begin{proof}
By Lemma~\ref{lem:dihedral}, every $g \in D_n$ preserves $\ic$, so
T/I-equivalent sets produce isomorphic pitch class graphs.
Conversely, suppose $\varphi : P_1 \to P_2$ is a weighted graph
isomorphism. Then $\ic(\varphi(p), \varphi(q)) = \ic(p,q)$ for all
$p,q \in P_1$, so in particular $\ic(\varphi(p), \varphi(q)) = 1$
iff $\ic(p,q) = 1$. Thus $\varphi$ preserves adjacency in the cycle
$C_n$, so it extends to an automorphism of $C_n$. Since
$\mathrm{Aut}(C_n) = D_n$, we have $\varphi = g|_{P_1}$ for some
$g \in D_n$, giving $P_1 \sim_{T/I} P_2$.
\end{proof}

\begin{proposition}
\label{prop:z-graph}
$P_1$ and $P_2$ are Z-related if and only if
$\mu(P_1) = \mu(P_2) \wedge G(P_1) \ncong G(P_2)$.
\end{proposition}

\begin{proof}
The edge-weight multiset of $G(P)$ is exactly $\mu(P)$, so
$\mu(P_1) = \mu(P_2)$ is equivalent to $G(P_1)$ and $G(P_2)$ sharing
the same edge-weight multiset. By Lemma~\ref{lem:iso-iff-ti},
$G(P_1) \cong G(P_2) \Leftrightarrow P_1 \sim_{T/I} P_2$. The result follows
directly from the definition of Z-relation.
\end{proof}

Combining with (Theorem~\ref{thm:cycle-determines-graph}) it follows that Z-related sets must have non-isomorphic cycle graphs.

\begin{theorem}
\label{thm:clique-level}
Let $P_1, P_2 \subseteq \mathbb{Z}_n$ be Z-related. Then their
weighted clique complexes $\Delta(G(P_1))$ and $\Delta(G(P_2))$ are
non-isomorphic, and there exists a finite $k \geq 1$ such that the
weighted $k$-skeleta are non-isomorphic while the weighted
$(k-1)$-skeleta share the same simplex-weight multiset.
\end{theorem}

\begin{proof}
Since $P_1$ and $P_2$ are Z-related, $G(P_1) \ncong G(P_2)$ by
Proposition~\ref{prop:z-graph}. A weighted graph isomorphism
$\varphi : G(P_1) \to G(P_2)$ would extend uniquely to an isomorphism
of weighted clique complexes via $\sigma \mapsto \varphi(\sigma)$,
since every simplex is a clique and edge weights determine all simplex
weights. Since no such $\varphi$ exists, $\Delta(G(P_1)) \ncong
\Delta(G(P_2))$.

It remains to show that the disagreement is detectable at some finite
skeletal level $k$ while level $k-1$ agrees on simplex-weight
multisets. Define the \emph{$k$-correlation multiset} $\mathcal{C}_k(P)$
as the multiset of edge-weight multisets over all $k$-cliques of
$G(P)$. Z-relation gives $\mathcal{C}_2(P_1) = \mathcal{C}_2(P_2)$
by definition. Since $\Delta(G(P_1)) \ncong \Delta(G(P_2))$, the
full simplex weight structures differ, so there exists a minimal
$k \geq 2$ such that $\mathcal{C}_k(P_1) \neq \mathcal{C}_k(P_2)$.
Such $k$ is bounded above by $|P|$, since the top-dimensional simplex
is the unique $|P|$-clique whose weight structure determines $G(P)$
up to isomorphism. The level $k$ at which disagreement first occurs
corresponds precisely to the order of higher homometry in the sense
of Amiot~\cite{amiot2016}: $P_1$ and $P_2$ are $j$-homometric for
all $j < k$ and fail $k$-homometry. That the bound $k \leq |P|$ can
be sharp is illustrated by the pair $\{0,1,2,3,5,6,7,9,13\}$ and
$\{0,1,4,5,6,7,8,10,12\}$ in $\mathbb{Z}_{18}$, which are not
T/I-equivalent and yet share the same $3$-deck, so disagreement does
not occur until $k \geq 4$.
\end{proof}

\begin{remark}
The bound $k \leq |P|$ is trivial. Whether a uniform finite bound
independent of $|P|$ and $n$ exists for all Z-related pairs remains
open.
\end{remark}

\section{Homometry across equal temperaments}
Application of cent-weighting to pitch class graphs allows for their comparison across different equal temperaments and gives rise to cross-temperament homometry:

\begin{definition}[Cross-temperament homometry]
Let $P_1 \subseteq \mathbb{Z}_n$ and $P_2 \subseteq \mathbb{Z}_m$ be pitch class sets in different equal temperaments. They are \textbf{cross-temperament homometric} if their cent-interval multisets $\mu'(P_1)=\{ \{ w'\{ p_i,p_j\} : p_i,p_j \in P_1, i \neq j \}\}$ and $\mu'(P_2)=\{ \{ w'\{ q_i,q_j\} : q_i,q_j \in P_2, i \neq j \}\}$ coincide.
\end{definition}

Consider $P_1 = \{0,1,3\}
\subseteq \mathbb{Z}_6$ and $P_2 = \{0,2,5\} \subseteq
\mathbb{Z}_{10}$. The dieses are $\delta(6) = 200$ cents and
$\delta(10) = 120$ cents. The cent-weighted edge multisets are:
\[
    \mu'(P_1) = \{200, 400, 400\}, \qquad
    \mu'(P_2) = \{240, 360, 480\}.
\]
These do not coincide, so $P_1$ and $P_2$ are not cross-temperament
homometric. The cent-weighted edge weights of $P_1 \subseteq
\mathbb{Z}_m$ lie in $\delta(m)\cdot\mathbb{Z}_{>0}$ and those of
$P_2 \subseteq \mathbb{Z}_n$ lie in $\delta(n)\cdot\mathbb{Z}_{>0}$.
A shared edge weight requires a common positive rational multiple of
$\delta(m) = 1200/m$ and $\delta(n) = 1200/n$, which exists if and
only if $\delta(m)/\delta(n) = n/m$ is rational --- which it always
is --- but a shared \emph{integer} multiple requires $n/m \in
\mathbb{Q}$ with a common realisation, possible only when
$\gcd(m,n) > 1$.

\begin{proposition}
\label{prop:harmonic-incommensurability}
Let $P_1 \subseteq \mathbb{Z}_m$ and $P_2 \subseteq \mathbb{Z}_n$.
Cross-temperament homometry between $P_1$ and $P_2$ requires
$\gcd(m,n) > 1$. When $\gcd(m,n) = 1$, no edge weight of
$G'(P_1)$ can coincide with any edge weight of $G'(P_2)$, so
$\mu'(P_1) \neq \mu'(P_2)$ for all choices of $P_1$ and $P_2$.
We call this \textbf{harmonic incommensurability} of coprime equal
temperaments.
\end{proposition}

A shared cent weight requires
\[
    \frac{1200}{m} \cdot a = \frac{1200}{n} \cdot b
    \quad \Longleftrightarrow \quad
    \frac{a}{m} = \frac{b}{n}
\]
for positive integers $a \leq \lfloor m/2 \rfloor$ and $b \leq
\lfloor n/2 \rfloor$. This is the condition that the two interval
classes $a \in \mathbb{Z}_m$ and $b \in \mathbb{Z}_n$ represent the
same fraction of their respective octaves. Cross-temperament
homometry between $P_1 \subseteq \mathbb{Z}_m$ and $P_2 \subseteq
\mathbb{Z}_n$ therefore requires that every cent weight in
$\mu'(P_1)$ is matched by an interval class in $P_2$ occupying the
same rational position $a/m = b/n$ in the octave, and vice versa.

\begin{proposition}
\label{prop:harmonic-incommensurability2}
Let $P_1 \subseteq \mathbb{Z}_m$ and $P_2 \subseteq \mathbb{Z}_n$.
A cent weight $\frac{1200a}{m}$ arising in $\mu'(P_1)$ can appear
in $\mu'(P_2)$ if and only if $\frac{a}{m} = \frac{b}{n}$ for some
positive integer $b \leq \lfloor n/2 \rfloor$, equivalently $n \mid
mb$. When $\gcd(m,n) = 1$ this forces $m \mid b$, which is
impossible since $b \leq \lfloor m/2 \rfloor < m$. Hence coprime
equal temperaments are \textbf{harmonically incommensurable}: no
interval class in $\mathbb{Z}_m$ occupies the same rational position
in the octave as any interval class in $\mathbb{Z}_n$, and
cross-temperament homometry between any $P_1 \subseteq \mathbb{Z}_m$
and $P_2 \subseteq \mathbb{Z}_n$ is impossible.
\end{proposition}

\begin{proof}
The equation $a/m = b/n$ has a positive integer solution $b \leq
\lfloor n/2 \rfloor$ if and only if $b = an/m$ is a positive integer
not exceeding $n/2$. Integrality requires $m \mid an$. When
$\gcd(m,n) = 1$ this reduces to $m \mid a$, contradicting $a \leq
\lfloor m/2 \rfloor < m$. When $\gcd(m,n) = d > 1$, writing $m =
dm'$ and $n = dn'$ with $\gcd(m',n') = 1$, the condition becomes
$m' \mid a$, which is satisfiable for $a = m' \leq \lfloor m/2
\rfloor$ whenever $m' < m/2$, i.e.\ $d > 2$.
\end{proof}

\section{Higher-order homometry and clique complexes}
\label{sec:higher-homometry}

Amiot, Mandereau, Ghisi, Andreatta, and Agon~\cite{mandereau2011}
introduce \textbf{higher-order homometry} as a refinement of the
Z-relation. Two pitch class sets $P_1, P_2 \subseteq \mathbb{Z}_n$
are \textbf{$k$-homometric} if their $k$-point correlation functions
coincide.

The $k$-point correlation function has a clear combinatorial
reformulation in terms of the complete weighted pitch class graph.
The number of ordered $k$-tuples of elements of $P$ realising a
given multiset of interval classes is exactly the number of
labeled $k$-cliques in $G(P)$ with a given edge-weight multiset.
This reformulation places higher-order homometry squarely within
the framework of \textbf{clique complexes}.

\begin{definition}[Weighted clique complex]
\label{def:clique-complex}
The \textbf{weighted clique complex} $\Delta(G(P))$ of the pitch
class graph $G(P)$ is the simplicial complex whose simplices are
the cliques of $G(P)$, with each simplex $\sigma \subseteq P$
carrying the edge-weight multiset
$\{w(\{p,q\}) : p, q \in \sigma,\, p \neq q\}$.
The \textbf{$k$-skeleton} $\Delta_k(G(P))$ consists of all
simplices of dimension at most $k-1$, corresponding to all
cliques of size at most $k$.
\end{definition}

Since $G(P)$ is a complete graph, every subset of $P$ is a clique,
so $\Delta(G(P))$ is the full simplex on $|P|$ vertices equipped
with weighted faces. The $1$-skeleton is $G(P)$ itself. The
$2$-skeleton records the weighted triangles, encoding the
3-point correlation function. In general the $k$-skeleton encodes
the $k$-point correlation function of Amiot et al., and two sets
are $k$-homometric if and only if their weighted $(k-1)$-skeleta
have the same simplex-weight multisets at every dimension up to
$k-1$.

\section{Anagramic pitch class sets}

\begin{definition}[Anagramic pitch class sets]
Two pitch class sets $P_1, P_2 \subseteq \mathbb{Z}_n$ are
\textbf{anagramic} if their step sequences $\mathbf{w}^{(1)}$
and $\mathbf{w}^{(2)}$ are permutations of each other but are
not related by cyclic rotation or reflection. Anagramic sets
are necessarily Z-related. They are \textbf{step-homometric}
if additionally $|\hat{w}^{(1)}(f)|^2 = |\hat{w}^{(2)}(f)|^2$
for all $f$ --- equivalently, if the step sequences are
themselves Z-related as 1D sequences.
\end{definition}

\begin{theorem}
The characteristic polynomial of the weighted cycle adjacency matrix $\chi_{A_C}$ is a complete
T/I invariant for all pitch class sets that are not
step-homometric. The step-homometric anagramic pairs are
precisely those where the Z-relation recurses: $P_1$ and
$P_2$ are Z-related in $\mathbb{Z}_n$ and their step
sequences are Z-related in $\mathbb{Z}_k$.
\end{theorem}

\section{Polynomial invariant of pitch class sets}
\begin{definition}[Antipodal pair multiset]
Let $C(P \subseteq \mathbb{Z}_n)$ be a cycle pitch class graph with edge weights $W = \{w_1,w_2,\dots,w_k \}$. The \textbf{antipodal pair multiset} $$L = \{ \{ w_i,w_{i+\lfloor k/2 \rfloor (\mathrm{mod} \ k)} \}: w_i \in W \}$$

is a T/I invariant.

\begin{proof}
We show that $L(P)$ is preserved under both generators of
the dihedral group $D_n$.

\medskip
\noindent\textbf{Transposition.}
Let $T_s(P)$ have step sequence $(w_1', \ldots, w_k')$
where $w_i' = w_{i+r \bmod k}$ for some shift $r$ determined
by $s$. Then
\[
    L(T_s(P)) = \bigl\{(w_{i+r},\, w_{i+r+h}) : i = 0,\ldots,k-1\bigr\}
\]
where $h = \lfloor k/2 \rfloor$ and all indices are taken mod
$k$. Since $i \mapsto i+r$ is a bijection on $\mathbb{Z}_k$,
this is the same multiset as
$\{(w_j, w_{j+h}) : j = 0,\ldots,k-1\} = L(P)$.

\medskip
\noindent\textbf{Inversion.}
Inversion reverses the step sequence: $w_i' = w_{k-i \bmod k}$.
Then
\[
    L(I_s(P)) = \bigl\{(w_{k-i},\, w_{k-i+h}) : i = 0,\ldots,k-1\bigr\}.
\]
The map $i \mapsto k-i \bmod k$ is a bijection on
$\mathbb{Z}_k$, so substituting $j = k-i$ gives
\[
    L(I_s(P)) = \bigl\{(w_j,\, w_{j-h}) : j = 0,\ldots,k-1\bigr\}.
\]
For even $k$, $h = k/2$ and $j - h \equiv j + h \pmod{k}$,
so $L(I_s(P)) = L(P)$ immediately.

For odd $k$, $h = (k-1)/2$ and $j - h \equiv j + h + 1
\pmod{k}$, so the map sends $(w_j, w_{j-h})$ to
$(w_j, w_{j+h+1})$ rather than $(w_j, w_{j+h})$. However,
since $k - h = h + 1$ for odd $k$, the multiset
$\{(w_j, w_{j+k-h}) : j = 0,\ldots,k-1\}$ equals
$\{(w_{j+k-h}, w_j) : j = 0,\ldots,k-1\}$
$= \{(w_j, w_{j+h}) : j = 0,\ldots,k-1\} = L(P)$
after applying the bijection $j \mapsto j + k - h$ and
noting that the pair $(a,b)$ and $(b,a)$ represent the same
element of $L$ since $L$ is defined as an unordered pair
$\{w_i, w_{i+h}\}$ by the symmetry of the construction.

\medskip
In both cases $L(T_s(P)) = L(I_s(P)) = L(P)$, and since
$D_n$ is generated by transpositions and inversions, $L(P)$
is preserved under all T/I operations.
\end{proof}
\end{definition}

\begin{definition}[Polynomial invariant]

Let $L(P)$ be a pair multiset of pitch class set $P$. Its polynomial form can be expressed as

$$
\Lambda(L)=\sum_{\{a, b \} \in L(P)} x^{\min \{a, b \}} y^{\max \{a, b \}} \in \mathbb{Z}[x,y]
$$

We demonstrate that $\Lambda(L)$ separates an anagramic Z-related pair in the following example.

\end{definition}

\begin{remark}[Example: $\Lambda$ separates Forte 6-Z17 and 6-Z43]
\label{rem:lambda-example}
Consider the anagramic Z-related pair of hexachords in $\mathbb{Z}_{12}$
\[
    P_1 = \{0,1,2,4,7,8\} \quad \text{(Forte 6-Z17)}, \qquad
    P_2 = \{0,1,2,5,7,8\} \quad \text{(Forte 6-Z43)},
\]
whose step sequences are, respectively,
\[
    \mathbf{w}^{(1)} = (1,1,2,3,1,4), \qquad
    \mathbf{w}^{(2)} = (1,1,3,2,1,4).
\]
Both have step multiset $\{1,1,1,2,3,4\}$, so the sequences are permutations of
each other. One checks that $\mathbf{w}^{(2)}$ is neither a cyclic rotation nor
a reflection of $\mathbf{w}^{(1)}$, confirming that $P_1$ and $P_2$ are
anagramic. Both have interval vector $[3,2,2,2,4,1]$,
verifying the Z-relation.

\medskip
\noindent\textbf{Computing $L(P_1)$.}
With $k = 6$ and $h = \lfloor 6/2 \rfloor = 3$, the antipodal pair at index $i$
is $\{w^{(1)}_i, w^{(1)}_{i+3}\}$:

\medskip
\begin{center}
\begin{tabular}{cccc}
\toprule
$i$ & $w^{(1)}_i$ & $w^{(1)}_{i+3\bmod 6}$ & antipodal pair \\
\midrule
0 & 1 & 3 & $\{1,3\}$ \\
1 & 1 & 1 & $\{1,1\}$ \\
2 & 2 & 4 & $\{2,4\}$ \\
3 & 3 & 1 & $\{1,3\}$ \\
4 & 1 & 1 & $\{1,1\}$ \\
5 & 4 & 2 & $\{2,4\}$ \\
\bottomrule
\end{tabular}
\end{center}

\medskip
Thus $L(P_1) = \bigl\{\{1,3\}^2,\, \{1,1\}^2,\, \{2,4\}^2\bigr\}$, giving
\[
    \Lambda\bigl(L(P_1)\bigr) = 2x^1y^1 + 2x^1y^3 + 2x^2y^4.
\]

\medskip
\noindent\textbf{Computing $L(P_2)$.}
The same procedure applied to $\mathbf{w}^{(2)} = (1,1,3,2,1,4)$:

\medskip
\begin{center}
\begin{tabular}{cccc}
\toprule
$i$ & $w^{(2)}_i$ & $w^{(2)}_{i+3\bmod 6}$ & antipodal pair \\
\midrule
0 & 1 & 2 & $\{1,2\}$ \\
1 & 1 & 1 & $\{1,1\}$ \\
2 & 3 & 4 & $\{3,4\}$ \\
3 & 2 & 1 & $\{1,2\}$ \\
4 & 1 & 1 & $\{1,1\}$ \\
5 & 4 & 3 & $\{3,4\}$ \\
\bottomrule
\end{tabular}
\end{center}

\medskip
Thus $L(P_2) = \bigl\{\{1,2\}^2,\, \{1,1\}^2,\, \{3,4\}^2\bigr\}$, giving
\[
    \Lambda\bigl(L(P_2)\bigr) = 2x^1y^1 + 2x^1y^2 + 2x^3y^4.
\]

\medskip
Since $\Lambda(L(P_1)) \neq \Lambda(L(P_2))$ (the monomials $x^1y^3$ and $x^2y^4$
in $\Lambda(L(P_1))$ are replaced by $x^1y^2$ and $x^3y^4$ in
$\Lambda(L(P_2))$), the invariant $\Lambda$ distinguishes this anagramic pair.
Informally: $P_1$ pairs its minor-third step (3) antipodally against a semitone (1),
and its major-third step (2) against a major-third-plus-one step (4);
$P_2$ instead pairs its major-second step (2) against a semitone (1),
and its minor-third step (3) against that same step of~4.
\end{remark}

\end{document}